\def\arxive{1}

\if\arxive1
	\documentclass[lettersize,journal,twocolumn,10pt]{article}
	\usepackage[a4paper, total={7in, 9in}]{geometry}
	\usepackage{mathptmx}
\else
	\documentclass[lettersize,journal]{IEEEtran}
\fi

\usepackage{amsmath,amsfonts}
\usepackage{algorithmic}
\usepackage{algorithm}
\usepackage{array}
\usepackage{subfig}
\usepackage{textcomp}
\usepackage{stfloats}
\usepackage{url}
\usepackage{verbatim}
\usepackage{graphicx}
\usepackage{cite}

\usepackage[usenames,dvipsnames]{xcolor}
\usepackage{mathtools}
\usepackage{amssymb}
\usepackage{amsthm}

\usepackage[hidelinks=True]{hyperref}
\usepackage{cleveref}
\usepackage{algorithmic}
\graphicspath{ {images/} }

\newcommand{\N}{\mathbb{N}}
\newcommand{\R}{\mathbb{R}}

\newcommand{\conv}[2]{#1\ast #2}

\DeclareMathOperator{\prox}{prox}
\DeclareMathOperator{\sign}{sign}
\DeclareMathOperator*{\argmin}{argmin}

\newtheorem{remark}{Remark}
\newtheorem{proposition}{Proposition}

\usepackage[normalem]{ulem}

\begin{document}

\title{Unsupervised energy disaggregation via convolutional sparse coding\thanks{\copyright~2023 IEEE, \url{https://ieeexplore.ieee.org/document/10288159}, \cite{aarset2023unsupervised}.}}

\author{Christian Aarset, Andreas Habring, Martin Holler, Mario Mitter
\thanks{The SOLGENIUM Project: Supported by the Austrian Research Promotion Agency (FFG) (Grant 881561).}%
\if\arxive0
\thanks{Manuscript received XXX YYY, 2022; revised AAA BBB, 2022.}
\fi}

\markboth{journal}%
{shell}

\if\arxive0
	\IEEEpubid{XXXX--XXXX/XX\$XX.XX~\copyright~2022 IEEE}
\fi

\maketitle

\begin{abstract}

In this work, a method for unsupervised energy disaggregation in private households equipped with smart meters is proposed. The method aims to classify power consumption as \emph{active} or \emph{passive}, granting the ability to report on the residents' activity and presence without direct interaction. This lays the foundation for applications like non-intrusive health monitoring of private homes. 

The proposed method is based on minimizing a suitable energy functional, for which the iPALM (inertial proximal alternating linearized minimization) algorithm is employed, demonstrating that various conditions guaranteeing convergence are satisfied. 

In order to confirm feasibility of the proposed method, experiments on semi-synthetic test data sets and a comparison to existing methods are provided.

\end{abstract}

\if\arxive0
	\begin{IEEEkeywords}
	Energy disaggregation, non-intrusive load monitoring, health monitoring, convolutional sparse coding
	\end{IEEEkeywords}
\fi

\section{Introduction} \label{sec:intro}
\if\arxive0
	\IEEEPARstart{S}{mart}
\else
Smart
\fi
meters -- energy meters capable of digital communication -- are becoming increasingly available in European domestic households. The regulation \emph{Intelligente Messger\"ate-Einf\"uhrungsverordnung – IME-VO} \cite{messgeraete_verordnung_source} derived from the EU directive \cite{messgeraete_eu_richtlinie_source} dictates that 95\% of all Austrian meter points have to be equipped with smart meters by the end of 2024. As a result, regular readings of household energy consumption will be available on a large scale, potentially enabling a variety of insights into power consumption, usage patterns and human activity. For an overview of potential health care applications of smart meters see \cite{fell2017energising,paxman2020smart}. As a lack of human activity in a household of, e.g., an elderly or care-dependent person over an extended period of time could be an indication of a potential emergency, in this work we focus on the possibility of detecting human activity, respectively a lack thereof, based on smart meter readings. Such an automatic detection would alleviate the need for the consumer to report incidents manually, and bypass the usage of intrusive surveillance devices.

Smart meters measure the aggregate energy consumption of a household rather than the consumptions of individual devices. Obtaining detailed information about the composition of the aggregate data is difficult, as it requires either directly measuring the energy consumption at many different points in the electrical circuit, or computational separation of the smart meter signal into distinct contributions, which is inherently ill-posed. For the latter problem, under the term computational \emph{energy disaggregation} or \emph{non-intrusive load monitoring} (NILM), there exists a variety of methods and we give an overview over existing literature in the following.

\subsection*{Related works}
Reviews of existing methods for NILM can be found in \cite{dash2022electric,gopinath2020energy}.
NILM methods are frequently split into the categories \emph{supervised} if they use labeled training data and \emph{unsupervised} in the converse case. While the latter case is substantially more difficult, it is also relevant for practical applications, due to the cost and difficulty of collecting sufficient amounts of labeled training data.
\paragraph{Supervised methods} One of the first approaches introduced in \cite{hart1992nonintrusive} is based on matching the changes in power consumption (event detection) to certain devices in the household. In \cite{he2016non}, the authors use graph signal processing to predict changes in the state of individual devices. There are many works on different versions of hidden Markov models (HMM), and in particular factorial versions thereof \cite{zoha2013low,bonfigli2017non,kong2016hierarchical,raiker2018approach}. In \cite{kolter2010energy,pandey2019structured}, the authors propose sparse coding for NILM, and in \cite{singh2017deep}, this approach is extended to a deep sparse coding model. Deep learning approaches have also been applied to NILM. In \cite{kaselimi2019multi}, the authors propose to use recurrent convolutional neural networks in a sequence-to-sequence fashion. In \cite{kaselimi2020energan}, the authors employ a sequence-to-sequence generative adversarial network (GAN) for energy disaggregation, and in \cite{kaselimi2020energan++}, this model is improved by using a deep learning recurrent classifier to model the discriminator component of the GAN. For more work on deep learning for NILM, see also \cite{li2021non,kaselimi2019bayesian,moradzadeh2021practical,
zhou2020sequence,murray2019transferability}.

\if\arxive0
	\IEEEpubidadjcol
\fi

\paragraph{Unsupervised methods}
By employing a priori assumptions on the device distributions HMMs can be used in an unsupervised manner \cite{makonin2015exploiting,johnson2013bayesian,
kolter2012approximate,pattem2012unsupervised,
kim2011unsupervised,aiad2016unsupervised,jia2015fully}. In \cite{parson2014unsupervised}, the authors propose a method of tuning pre-trained general appliance models of an HMM to devices of a specific household by only using aggregate data from this household. In \cite{liao2014non}, the authors propose a method comprised of event detection, feature extraction, and classification. The method is proposed in an unsupervised and a supervised manner, using a decision tree and dynamic time wrapping for classification, respectively. Similar strategies are considered in \cite{liu2019low} and \cite{liu2020secure}, and in \cite{zhao2016training}, such an approach is combined with graph signal processing methods. In \cite{miyasawa2019energy}, the authors propose to solve the NILM problem using non-negative matrix factorization, and also a shift-invariant version thereof.

\subsection*{Contribution}
In this work, we propose a convolutional sparse coding (CSC) approach similar to \cite{kolter2010energy,pandey2019structured}. Contrary to these works, however, our method completely omits any prior dictionary learning using labeled data. We believe this poses a clear benefit, as we are able to circumvent several crucial issues: As elaborated in the introduction of \cite{wiese2019open}, numerous difficulties are encountered in the acquisition, sharing, and reuse of household energy data, which leads to a rather limited amount of training data. Avoiding difficulties related to the scarcity of data by using a representative dictionary of general appliance models currently only works well for simple devices \cite[Section II, D]{makonin2015exploiting} especially since a sufficiently extensive dictionary is not yet available \cite[Section 3.5.7]{najafi2018data}. Obtaining a dictionary by measuring one activity cycle of each electrical device in the household, on the other hand, poses an excessive and intrusive installation effort, and is consequently undesirable. Altogether, we therefore believe that methods not relying on training data are crucially needed for NILM. Moreover, our experiments suggest that, even when training data is available, our unsupervised approach is more robust compared to supervised ones with respect to shifts in distribution between training and testing data. (see \Cref{table:unsupervised_activity}, \Cref{fig:activity} and also \cite{shi2022robust,liu2021self}). 

Motivated by the aforementioned monitoring applications, in this work we will focus on experimental results for activity prediction, but provide also results for full disaggregation into individual devices. For activity prediction, we disaggregate into an \emph{active} and a \emph{passive channel}, where the presence of the former will be used as an indicator of human activity.

\section{Modeling and Problem formulation}
The power consumption $u\in\R^m$ of a household, with $m\in\N$ the number of evenly spaced time points, is given as the sum of the power consumptions of all individual devices $u^i\in\R^m$ for $i=1,2,\dots N$ contained in this household as
\[
u = \sum\limits_{i=1}^N u^i + \epsilon,
\]
where $\epsilon$ denotes measurement noise. The goal of NILM is to recover all $u^i$ given only knowledge of $u$. Due to the severe ill-posedness of this problem, additional modeling assumptions have to be made in order to obtain reasonable estimates of the $u^i$. Specifically, our model builds on the observation that the main contributors to a household's energy consumption consist of disparate device power consumption peaks, or simply \emph{device characteristics}. That is, whenever a device is used for a fixed period of time, its individual power consumption over time will consistently be of the same approximate shape and magnitude. Let for instance $p\in\R^{2q+1}$ be one activation cycle of length $2q+1$ of a dishwasher. Assume at each of the times tamps $t_1, t_2,\dots$ the dishwasher is activated for one cycle. Then, via the convolution, the dishwasher's power consumption can be computed as
\[c*p := \Bigl(\sum_{\Delta q = -q}^q \tilde{c}_{j+q+\Delta q}\tilde{p}_{q+1-\Delta q}\Bigr)_{j=1}^m \in \R^m\]
where $c = (0,\dots,0,1,0,\dots,0, 1,0,\dots,0)\in\R^{m+2q}$ with the ones appearing at indices $t_1,t_2,\dots$. Figuratively speaking, the \emph{coefficient} $c$ places the activation $p$ at the respective times. Inspired by this we model each device $u^i$ as a convolution $c^i\ast p^i$ -- or a sum thereof for devices with multiple different characteristic activation cycles -- of an \emph{atom} $p^i\in\R^{2q+1}$ resembling one activation cycle of device $i$ and a \emph{coefficient} $c^i\in\R^{m+2q}$ indicating the times of activation of device $i$.
Explicitly, this leads to the following formulation:
\[
u = \sum\limits_{i=1}^N c^i\ast p^i + \epsilon
\]
For the sake of simplifying notation we also denote the sum-of-convolutions map for multiple devices $c\ast p := \sum_{i=1}^N c^i\ast p^i$.
In the application of activity detection, we are particularly interested in the case where we have several \emph{active channels}, representing the characteristics of all devices that require human interaction -- thus indicating human activity -- and one \emph{passive channel}, consisting of the characteristics of all remaining devices. Examples of devices contributing to the active channels include, but are not limited to, ovens, washing machines and hair dryers, while devices such as fridges and routers might contribute to the passive channel. Devices in the active channels are assumed to be needed only a few times in the span of a single day, in the sense that the number of times they are activated each day is relatively low. This leads to a natural sparsity of incidents of each device characteristic on the temporal axis (see \Cref{fig:activity}, top right). In summary we desire the following properties for our atoms and coefficients:

\begin{itemize}

\item The atoms should be bounded, so that scaling is left to the coefficients. Atoms and coefficients should be positive.%

\item With $(c^1,p^1)$ representing the passive channel, we let all remaining pairs, that is, $(c^i,p^i)$ for $i\geq 2$, serve to explain the active channels; thus, active energy consumption is approximated by $\sum_{i=2}^Nc^i\ast p^i$. The coefficients $c^i$, $i\geq 2$ will be penalized in $\|\cdot\|_1$-norm, as this norm emulates the previously mentioned temporal sparsity (see \cite{hastietibshiraniwainwright2015statsparse}), while remaining analytically tractable. However, as some passive devices, such as a router, may consume power throughout most of the day, a sparsity-enforcing penalty term for $c^1$ becomes undesirable. Accordingly, we penalize $c^1$ in the $\|\cdot\|_2$-norm.

\end{itemize} 

Combining the above leads to the minimization problem
\begin{equation}\label{eq:minimization}
\begin{aligned}
	\min_{c,p} & \|u - \conv{c}{p}\|^2_2 + \lambda_{\mathfrak{p}}\|c^1\|_2 + \lambda_{\mathfrak{a}}\sum_{i=2}^N\|c^i\|_1 \\
	& \text{s.t. }  c\geq 0,\,  p\geq 0,\, \|p^i\|_2\leq 1 \text{ for all }i=1,\ldots,N,
\end{aligned}
\end{equation}
with $\lambda_{\mathfrak{p}}$, $\lambda_{\mathfrak{a}}>0$; the positivity constraints $c\geq 0$ and $p\geq 0$ are meant in the componentwise sense. The first term in \eqref{eq:minimization} ensures the solution to approximately solve the energy disaggregation problem and the remaining terms incorporate constraints and ensure well-posedness.

\section{Inertial Proximal Alternating Linearized Minimization}

As \eqref{eq:minimization} does not generally allow for an analytical solution, we require an appropriate algorithm to numerically solve the minimization problem. In that respect, we face two main difficulties with our problem formulation: It is neither jointly convex nor differentiable in $(c,p)$, meaning standard convex or gradient descent-based algorithms cannot be used as is. To overcome this, we employ the inertial Proximal Alternating Linearized Minimization (iPALM) algorithm as proposed in \cite{pock2016ipalm}, which is suitable for nonconvex, nonsmooth optimization problems of a particular structure. Therefore, we first rewrite \eqref{eq:minimization} into a form where the iPALM algorithm is applicable, i.e., we consider
\begin{equation}\label{eq:minimization_reformulated}
	\min_{c,p} F(c,p), \text{ where } F(c,p) := R(c,p) + f_1(c) + f_2(p).
\end{equation}
Here, $R(c,p):= \|u - \conv{c}{p}\|^2_2$,
\begin{align*}
	f_1(c) & := \begin{cases}\lambda_{\mathfrak{p}}\|c^1\|_2 + \lambda_{\mathfrak{a}}\sum_{i=2}^N\|c^i\|_1, & c\geq 0, \\ \infty, & \text{otherwise,}\end{cases} \\
	f_2(p) & := \begin{cases}0, & p\geq 0 \text{ and } \max_{1\leq i\leq N}\|p^i\|_2\leq 1, \\ \infty, & \text{otherwise.}\end{cases}
\end{align*}

To ensure convergence of iPALM according to \cite{pock2016ipalm}, several properties of the functionals involved in \eqref{eq:minimization_reformulated} need to be satisfied. It is clear that $f_1$, $f_2$ and $R$ are proper, lower semi-continuous, bounded from below, and that $R$ is differentiable. Moreover, $f_1$ and $f_2$ are both convex, with $f_1$ being a linear combination of componentwise norms under a convex constraint, and $f_2$ simply being a convex constraint. We further need Lipschitz continuity of the derivative of $R$ as stated in the following proposition.

\begin{proposition}\label{prop:ipalm_main_conditions} For $R$ as above, the following holds:

\begin{enumerate}

\item \label{thrm:ipalm_main_conditions:lipschitz} For any fixed $p\in(\R^{2q+1})^N$, $c\mapsto \nabla_cR(c,p)$ is Lipschitz continuous. Similarly, for any fixed $c\in(\R^{m+2q})^N$, $p\mapsto \nabla_pR(c,p)$ is Lipschitz continuous.

\item \label{thrm:ipalm_main_conditions:lipschitz_on_bounded} $\nabla R$ is Lipschitz continuous on bounded subsets of $(\R^{m+2q})^N\times(\R^{2q+1})^N$.

\end{enumerate}

\end{proposition}

\begin{proof}

Write $R(c,p) := \|\cdot \|^2_2 \circ \left[u - \cdot\right] \circ (c\ast p)$. As the map $c\mapsto c\ast p$ is linear for fixed $p$, one has $c\ast p = P(p)c$ with $P:(\R^{2q+1})^N\to \R^{m\times N(m+2q)}$. Simultaneously, for fixed $c$, $c\ast p = C(c)p$ with $C:(\R^{m+2q})^N\to \R^{m\times N(2q+1)}$. Thus,
\begin{align*}
	\nabla_c R(c,p) & = -2P(p)^T[u - c\ast p], \\
	\nabla_p R(c,p) & = -2C(c)^T[u - c\ast p].
\end{align*}
It follows that for any fixed $p\in(\R^{2q+1})^N$ and any $c$, $\tilde{c}\in(\R^{m+2q})^N$, one has
\begin{equation}\label{thrm:ipalm_main_conditions:eq:lipschitz_constant_c}
\begin{aligned}
	 \|\nabla_cR(c,p) - \nabla_cR(\tilde{c},p)\|_2
	\leq &2\|P(p)\|^2\|c - \tilde{c}\|_2 \\
	 \\ \eqqcolon &L_1(p)\|c-\tilde{c}\|_2
\end{aligned}
\end{equation}
and analogously, for all $c\in(\R^{m+2q})^N$ and all $p$, $\tilde{p}\in(\R^{2q+1})^N$,
\begin{equation}\label{thrm:ipalm_main_conditions:eq:lipschitz_constant_p}
\begin{aligned}
	\|\nabla_pR(c,p) - \nabla_pR(c,\tilde{p})\|_2
	\leq &2\|C(c)\|^2\|p - \tilde{p}\|_2 \\
	\eqqcolon &L_2(c)\|p-\tilde{p}\|_2,
\end{aligned}
\end{equation}
showing \ref{thrm:ipalm_main_conditions:lipschitz}). 

To show \ref{thrm:ipalm_main_conditions:lipschitz_on_bounded}), we note that as the maps $p\mapsto P(p)$, $p\mapsto P(p)^T$, $c\mapsto C(c)$ and $c\mapsto C(c)^T$ are linear and bounded, it is clear that $\nabla_cR(c,p)$ and $\nabla_pR(c,p)$ are differentiable with respect to $(c,p)$ jointly, with bounded derivatives on bounded subsets of $(\R^{m+2q})^N\times(\R^{2q+1})^N$. The mean value inequality immediately yields Lipschitz continuity on bounded sets.

\end{proof}

\begin{remark}
The preceding theorem ensures that Assumption A of \cite{pock2016ipalm} is satisfied (note that the lower bound in condition (iv) of this assumption can be easily satisfied by choosing sub-optimal Lipschitz constants that are bounded from below by a positive number). Further, using results from \cite{bolte2014proximal}, it is clear that the objective function in \eqref{eq:minimization_reformulated} is semi-algebraic; it thus satisfies the Kurdyka-\L{}ojasiewicz property. Hence, under suitable upper bounds on the stepsizes (see Assumption B of \cite{pock2016ipalm}), the iPALM algorithm can be applied to \eqref{eq:minimization_reformulated}, and it is guaranteed that any bounded sequence of iterates converges to a critical point of \eqref{eq:minimization_reformulated} (\!\!\cite[Theorem 4.1]{pock2016ipalm}).
\end{remark}

\begin{remark}\label{remark:lipschitz_bounds}

The Lipschitz constants from \ref{thrm:ipalm_main_conditions:lipschitz}) in Proposition \ref{prop:ipalm_main_conditions} are needed for the stepsize choice in the iPALM algorithm. Those correspond to the operator norms $\|P(p)\|$ resp. $\|C(c)\|$, which can be computed explicitly via spectral radii, but this may prove impractical for numerics. Denoting the matrix representation of $c^i\mapsto c^i*p^i$ as $P^i(p^i)$, norm estimates on the $P^i(p^i)$ can be obtained as follows: With $J:\R^{m+2q}\to\ell^2(\R)$, $(Jx)_i := x_i$ for $1\leq i\leq m+2q$ and $(Jx)_i := 0$ otherwise, with $K:\R^{2q+1}\to\ell^1(\R)$, $(Kx)_i := x_{i+2q+1}$ for $-2q\leq i\leq 0$ and $(Kx)_i := 0$ otherwise, and with $\tilde\ast$ the discrete convolution between $\ell^2(\R)$ and $\ell^1(\R)$, Young's inequality implies
\begin{align*}
	\|P^i(p^i)c^i\|_2 = \|c^i\ast p^i\|_2 \leq \|Jc^i\tilde\ast Kp^i\|_{\ell^2} \\
	\leq \|Jc^i\|_{\ell^2}\|Kp^i\|_{\ell^1} = \|c^i\|_2\|\|p^i\|_1.
\end{align*}
Taking the supremum yields $\|P^i(p^i)\|\leq\|p^i\|_1$. A similar argument shows $\|C^i(c^i)\|\leq\|c^i\|_1$ for $C^i(c^i)$ the matrix representation of $p^i\mapsto c^i*p^i$. Applying standard matrix norm estimates, one altogether has
\begin{align*}
	\|C^i(c^i)\| \leq 
	\min\Biggl\{\Biggr.
	\|c^i\|_1, 
	\sqrt{2q+1}\max_{0\leq\Delta j\leq 2q}\sum_{j=1}^mc^i_{j+\Delta j}, \\
	\sqrt{m}\max_{0\leq\Delta j\leq m-1}\sum_{j=1}^{2q+1}c^i_{j+\Delta j},
	\left(\sum_{j=1}^{m}\sum_{\Delta j=0}^{2q}|c^i_{j+\Delta j}|^2\right)^{\tfrac{1}{2}}\Biggl.\Biggr\},
\end{align*}
where it can be seen that each of these estimates may be optimal under various circumstances. Note also that although similar estimates are available for $P^i(p^i)$, its particular structure causes these to all be bounded below by $\|p^i\|_1$. Norm estimates for $P(p)$, respectively $C(c)$ are obtained using that, by Cauchy-Schwarz's inequality, for $v\in(\R^{m+2q})^N$, $\|P(p)v\|_2  \leq \sqrt{\sum_{i=1}^N\|P^i(p^i)\|^2}\|v\|_2$ and, analogously, 
$\|C(c)\| \leq \sqrt{\sum_{i=1}^N\|C^i(c^i)\|^2}$.
\end{remark}
Applied to our setting, the iPALM algorithm takes the following form: Fix a suitable non-zero initialisation $(c_{-1},p_{-1})=(c_0,p_0)\in (\R^{m+2q})^N \times (\R^{2q+1})^N$, all non-negative componentwise. Fix $\varepsilon \in(0,1)$ small, and for $i\in\{1,2\}$, choose $\bar{\alpha}_i\in(0,1-\varepsilon)$, $\bar{\beta}_i>0$ arbitrary. Define $\gamma_i := \frac{\bar{\alpha}_1 + 2\bar{\beta}_1}{2(1 - \varepsilon - \bar{\alpha}_1)}\lambda_i^+$, where $\lambda_1^+$ resp. $\lambda_2^+$ corresponds to the upper bound over all $k$ on $L_1(p_k)$ as in \eqref{thrm:ipalm_main_conditions:eq:lipschitz_constant_c}, resp. $L_2(c_k)$ as in \eqref{thrm:ipalm_main_conditions:eq:lipschitz_constant_p}; when such a bound is not analytically apparent, a viable substitute in the $k$-th iteration step is $\max_{1\leq j\leq k}L_1(p_j)$ resp. $\max_{1\leq j\leq k+1}L_2(c_j)$, possibly dropping early indices to avoid excessive size. For each $k\in\N_0$, repeat the following process:
\begin{center}
\fbox{
\begin{minipage}{0.36\textwidth}
\begin{algorithmic}
\REQUIRE Pick $\alpha_1^k\in[0,\bar{\alpha}_1]$, $\beta_1^k\in[0,\bar{\beta}_1]$.
\STATE $y_1^k := c_k + \alpha_1^k(c_k - c_{k-1})$
\STATE $z_1^k := c_k + \beta_1^k(c_k - c_{k-1})$
\STATE $\tau_1^k := \tfrac{(1+\varepsilon)\gamma_1 + (1 + \beta_1^k)L_1(p_k)}{2 - \alpha_1^k}$
\STATE $c_{k+1} \in \prox_{\tau_1^k}^{f_1}\left(y_1^k - \tfrac{1}{\tau_1^k}\nabla_cR(z_1^k,p_k)\right)$
\REQUIRE Pick $\alpha_2^k\in[0,\bar{\alpha}_2]$, $\beta_2^k\in[0,\bar{\beta}_2]$.
\STATE $y_2^k := p_k + \alpha_2^k(p_k - p_{k-1})$
\STATE $z_2^k := p_k + \beta_2^k(p_k - p_{k-1})$
\STATE $\tau_2^k := \tfrac{(1+\varepsilon)\gamma_2 + (1 + \beta_2^k)L_2(c_{k+1})}{2 - \alpha_2^k}$
\STATE $p_{k+1} \in \prox_{\tau_2^k}^{f_2}\left(y_2^k - \tfrac{1}{\tau_2^k}\nabla_pR(c_{k+1},z_2^k)\right)$
\end{algorithmic}
\end{minipage}
}
\end{center}

The mappings $\prox_{\tau_1^k}^{f_1}$ and $\prox_{\tau_2^k}^{f_2}$ above are proximal mappings (see \cite[Chapter 6]{prox}), which are given explicitly as follows.
\begin{proposition}

For $t>0$, the functions $f_1$, $f_2$ as above satisfy
\begin{align*}
	(\prox_t^{f_1}(c))^i & = \begin{cases}
	\Bigl(1 - \frac{\lambda_{\mathfrak{p}}/t}{\max\left\{\|[c^1]_+\|_2,\lambda_{\mathfrak{p}}/t\right\}}\Bigr)[c^1]_+, & i = 1, \\
	\left[c^i - \tfrac{\lambda_{\mathfrak{a}}}{t}\right]_+, & i > 1,
	\end{cases} \\
	(\prox_t^{f_2}(p))^i & = \frac{[p^i]_+}{\max\left\{\|[p^i]_+\|_2,1\right\}},
\end{align*}
where $[\cdot]_+$, $|\cdot|$, $-$ and $\sign$ are understood componentwise.

\end{proposition}

\begin{proof}

Denote for any set $E$ by $\delta_E$ the indicator function satisfying $\delta_E(x) = 0$ for $x\in E$ and $\delta_E(x) = \infty$ for $x\notin E$, and by $\delta_+$ the indicators for the positive cones. $P_E$ denotes projection onto $E$. 

We begin by noting that \cite[Theorem 6.6]{prox}, and the fact that $f_1$ and $f_2$ can both be written as componentwise sums enables us to compute each proximal mapping componentwise, with the $i$-th component depending only on $c^i$ resp. $p^i$. 

\underline{$(\prox_t^{f_1}(c))^1$:} Writing it out explicitly, one has 
\begin{equation*}
\begin{aligned}
&(\prox^{f_1}_t(c))^1=\\
& \argmin_{d\in\R^{m+2q}}
 \sqrt{\sum_{i=1}^{m+2q}d_i^2} + \sum_{i=1}^{m+2q} \delta_+(d_i) + \frac{t}{2}|d_i-c_i^1|^2 .
\end{aligned}
\end{equation*}
From this, it is apparent that any solution $\hat{d}\in \R^{m+2q}$ will satisfy $\hat{d}_i = 0$ if $c^1_i \leq 0$. As such, one can replace $c^1$ by $[c^1]_+$ in the above minimization problem without changing the solution. But with $[c^1]_+$ instead of $c^1$, it is clear that any solution of the above problem will be non-negative even without the non-negativity constraint, making it equivalent to 
\begin{equation*}
\begin{aligned}
&\prox^{\lambda_{\mathfrak{p}}\|\cdot\|_2}_t([c^1]_+) =\\
&\argmin_{d\in\R^{m+2q}}
 \sqrt{\sum_{i=1}^{m+2q}d_i^2} + \frac{t}{2} \sum_{i=1}^{m+2q}   |d_i-([c^1]_+)_i|^2.
\end{aligned}
\end{equation*}
By \cite[Example 6.19]{prox}, the solution of the latter is of the claimed form.

\underline{$(\prox_t^{f_1}(c))^i$ for $i\geq 2$:} Writing $f_1(c)^i_j = |c^i_j| + \delta_+(c^i_j)$ for $1\leq j\leq N$, a combination of \cite[Example 6.8]{prox}  with \cite[Lemma 6.14(b)]{hollerreview}  immediately yields the claimed form.

\underline{$\prox_t^{f_2}(p)$:} Observe that $f_2(p) = \sum_{i=1}^N\delta_{C_i}(p^i)$, where $C_i := \left\{x\in\R^{2q+1} \mid x\geq 0, \, \|x\|_2\leq 1\right\}$. Employing \cite[Theorems 6.6, 6.24 and 6.30]{prox}, as well as noting that $C_i = \left\{x\in\R^{2q+1} \mid g_i(x) \leq 1\right\}$, where $g_i(x) := \delta_+(x) + \|x\|_2^2$ with effective domain $\R^{2q+1}_+$, one has 
\begin{equation*}
(\prox_t^{f_2}(p))^i = P_{C_i}(p^i) = \begin{cases}
[p^i]_+, & g_i([p^i]_+) \leq 1, \\
\prox_{\lambda g_i}(p^i), & \text{else,}
\end{cases}
\end{equation*}
where $\lambda>0$ solves 
\begin{align*}
1 & = g_i(\prox_{\lambda g_i}(p^i)) = g_i\bigl(\prox_{\tfrac{1}{2\lambda+1}\delta_+}(\frac{p^i}{2\lambda+1})\bigr) \\
& = g_i\Bigl(\frac{[p^i]_+}{2\lambda+1}\Bigr) = \frac{1}{(2\lambda+1)^2}\|[p^i]_+\|_2^2,
\end{align*}
with the last three equalities following from \cite[Theorems 6.13, 6.24 and Lemma 6.26]{prox}. Simplifying leads to the desired form.
\end{proof}

\section{Experiments}
In this section, we provide experimental results of the proposed method (CSC) for two different scenarios, namely activity prediction, where we try to detect whenever energy is consumed actively by a person, and multichannel disaggregation, where we aim to reconstruct all individual device signals that amount to an aggregate energy signal.
In particular, \Cref{table:unsupervised_activity} in \Cref{sec:exp_activity}, and \Cref{sec:exp_multi} contain results for the setting that no ground truth training data is available for the household of interest, constituting a realistic use case as explained in \Cref{sec:intro}.

\subsection{Data}

For our experiments, we employ the SynD data set generator \cite{klemenjak2020synthetic}, an algorithm to generate synthetic energy signals. This enables us to work with clean, yet realistic, data. The algorithm uses real device energy signatures, which it randomly alters and places throughout the day. It became apparent, however, that the fridge's energy signal generated with SynD is unrealistically regular, being identical each day; for this reason, we replaced the fridge by an energy signal taken from the UK-DALE 6 seconds data set \cite{UK-DALE}, and resampled to the desired sampling rate by weighted averaging. 

Explicitly, we generate two different synthetic households, referred to as synd1 and synd2. Each of the two data sets consists of four devices, synd1 of the fridge from UK-DALE's house 1 as well as SynD's electric space heater, washing machine, and hair dryer, and synd2 of the fridge from UK-DALE's house 2 and SynD's dishwasher, iron, and watercooker. We use a sampling rate of 1 minute, and the data sets are normalized to have values between 0 and 1.

\subsection{Performance metrics}

\paragraph{Binary perfomance metrics}
Both experiments -- activity prediction and multichannel disaggregation -- can be cast as binary classification tasks with multiple channels. At each time point each of the channels to predict has to be classified as either ON or OFF. In case of activity prediction, for each time point we aim to classify the \emph{active} signal as ON or OFF depending on whether a person is actively consuming energy or not, and similarly classify the \emph{passive} signal as ON or OFF depending on whether a \emph{passive device} is consuming energy or not. In case of multichannel disaggregation, we aim to classify each individual device as ON or OFF. Note that the proposed method, as well as the comparison methods, yield non-binary output signals, namely the actively consumed energy or the individual device energy signatures, rather than a binary ON/OFF signal. Thus, before evaluating classification performance, non-binary predictions are first transformed to binary signals via pointwise thresholding by 0.01 times the maximum value of the ground truth signal. Given $\hat{u},u\in\R^m$ a true non-binary signal and its prediction, respectively, we compute $\hat{U}$ and $U$ as $\hat{U} = (\hat{u}\geq 0.01 \|\hat{u}\|_\infty)$ and $U = (u\geq 0.01 \|\hat{u}\|_\infty)$ in Python notation. As a metric for evaluating binary classification, we use a modified Matthews correlation coefficient (MCC) and F1 score. Given $\hat{U},U\in \{0,1\}^m$ a true signal and its prediction, respectively, the MCC and F1 are defined as %
\begin{equation}\label{eq:mcc}
\begin{aligned}
\text{MCC}(\hat{U},U)\coloneqq\frac{(\text{TP})(\text{TN}) - (\text{FP})(\text{FN})}{ \sqrt{(\text{TP} + \text{FP})  (\text{TP} + \text{FN})  (\text{TN} + \text{FP})  (\text{TN} + \text{FN})}}
\end{aligned}
\end{equation}
and
\begin{equation}\label{eq:f1}
\begin{aligned}
\text{F1}(\hat{U},U)\coloneqq\frac{2\text{TP}}{2\text{TP}+\text{FP}+\text{FN}}
\end{aligned}
\end{equation}
where $\text{TP}=\hat{U} \cdot U$, $\text{TN}=(1-\hat{U})\cdot(1-U)$, $\text{FP}=(1-\hat{U})\cdot U$, and $\text{FN}=\hat{U}\cdot (1-U)$, with $\cdot$ the Euclidean scalar product, denote the number of true positives, true negatives, false positives, and false negatives, respectively. The MCC score is a value in $[-1,1]$, where $+1$ and $-1$ correspond to perfect agreement or disagreement, respectively, and $0$ indicates no relationship. The F1 score is a value in $[0,1]$, 1 being the best score. In view of relevant monitoring applications, where it is sufficient to predict activity up to a reasonable inaccuracy in time, we introduce an \emph{activity prediction radius} $R\in \mathbb{N}_0$ into the computation of the scores. We first define a modified true activity $\tilde{U}$ via $\tilde{U}_i = 1 $ if there exists $j\in \{1,2,\dots,m\}, |i-j|\leq R$ such that $\hat{U}_j=1$ and $\tilde{U}_i = 0 $ otherwise.
Hence, $\tilde{U}_i$ is equal to one whenever there was true activity within a $(2R+1)$ minute time window around $i$. We then compute $\text{TP} = \tilde{U}\cdot U$, $\text{TN} = (1-\hat{U})\cdot (1-U)$, $\text{FP} = (1-\tilde{U})\cdot U$, and $\text{FN} = \hat{U}\cdot (1-U)$. Using these values, the modified MCC and F1 are obtained via \eqref{eq:mcc} and \eqref{eq:f1} and denoted as $\text{MCC}_R$ and $\text{F1}_R$, respectively. In our experiments, we use $R=7$.

While both metrics are presented going forward, we generally place additional emphasis on the MCC, as it is more robust. Whereas the F1 score may falsely assign near-perfect scores by, e.g., classifying every instance as positive in a data set containing significantly less negatives than positives, the MCC inherently adjusts for such data set imbalance, see \cite{chiccojurman2020mcc}.

\paragraph{Non-binary performance metrics}
We do not only want to quantify whether the state of a device was correctly predicted, but also how well a method predicts the magnitude of its energy consumption. Thus, we also employ two non-binary metrics to evaluate performance, namely the normalized mean squared error (NMSE) and the normalized maximum absolute error (NMAE), defined as

\[\text{NMSE}(\hat{u},u) = \frac{\|\hat{u} - u\|_2 ^2}{\|\hat{u}\|_2^2}, \, 
\text{NMAE}(\hat{u},u) = \frac{\max_i |\hat{u}_i - u_i|}{\max_i |\hat{u}_i|}.\]

\subsection{Experimental setup}

The CSC experiments were carried out in the following manner: Given an aggregate energy signal $u\in\R^m$, initialize $(c_{-1},p_{-1}) = (c_0,p_0)\in(\R^{m+2q})^N\times(\R^{2q+1})^N$ by, e.g., selecting random non-zero features of $u$ of appropriate length as atoms and their location, shifted appropriately, as coefficients. Pick $\alpha_1^k=\alpha_2^k=\beta_1^k=\beta_2^k=0.7$ for all $k\in\N_0$, and compute the $\tau_1^k$ and $\tau_2^k$ in accordance with \cite{pock2016ipalm}, Remark 4.1. Employ the relative change  in the modified energy 
\begin{equation*}
	\Psi_k := F(c_k,p_k) + \frac{\gamma_1}{2}\|c_k-c_{k-1}\|_2^2 + \frac{\gamma_2}{2}\|p_k - p_{k-1}\|_2^2
\end{equation*}
(see \cite{pock2016ipalm}, (4.6)) as the stopping rule; explicitly, we stop when $\tfrac{|\Psi_k-\Psi_{k-1}|}{\Psi_{k-1}} < 10^{-6}$ or when $\Psi_k=0$. Across all experiments, the method terminates after around 3000-6000 iterations. One iteration takes roughly 0.55sec on an AMD Ryzen 7 3800X 8-Core Processor CPU.

\subsection{Results}
\subsubsection{Activity prediction}\label{sec:exp_activity}
In our data sets, the passive consumption is comprised solely of the fridge signal, while the active consumption consists of all remaining devices. We compare to four different supervised methods from the non-intrusive load monitoring toolkit NILMTK \cite{batra2014nilmtk,batra2019towards}, namely, combinatorial optimisation (CO), a factorial hidden Markov model (FHMM), a neural network model (S2S), and discriminative sparse coding (DSC), and to the fully unsupervised graph signal processing method (GSP) \cite{zhao2016training}, for which a Python implementation is available at \cite{gsp_git}. All methods are used to predict the active and the passive energy consumption. The data is normalized before plotting and computing the performance metrics. We report the scores on the active component of the signal in \Cref{table:supervised_activity,table:unsupervised_activity}. In \Cref{table:supervised_activity}, we show the testing scores for the case of training and testing on the same household. That is, all trainable algorithms use four weeks of the energy signal for supervised training, and a different selection of four weeks of the same energy signal for testing. In this setting, we observe that the proposed method performs comparably to the deep learning method S2S, and that they both outperform the remaining methods. For the scores in \Cref{table:unsupervised_activity}, on the other hand, we use 4 weeks of data of different households for training and testing. This corresponds to the use case, that no training data of the household of interest is available and, thus, training is conducted using accessible data of a different house, which leads to a distribution shift between training and testing data, especially since synd1 and synd2 even contain different devices. Empirically, the proposed method on average outperforms its (supervised) rivals in this setting.
While CSC also uses data to determine hyperparameters, it uses significantly less information for this process than the other methods, and is thus less prone to overfitting. That is, it is not necessary for CSC to learn, e.g., the distribution of individual devices (FHMM), the shapes of device signatures (DSC), or the parameters of a neural network etc., but solely to determine the hyperparameters $\lambda_\mathfrak{a}$, $\lambda_\mathfrak{p}$, $N$ and $q$ of our energy functional. We believe that this is the reason for CSC generalizing better than its supervised rivals to unseen data with a potentially slightly different distribution. In general, in \Cref{table:supervised_activity,table:unsupervised_activity} we observe that S2S yields the best MCC/F1 scores and CSC performs superior in terms of NMSE/NMAE scores. We believe that this relates to the fact that CSC recognizes the specific shapes of energy consumption more accurately. This is also supported by \Cref{fig:activity}. While FHMM and CO also closely fit the ground truth passive energy signal, CSC is the only method to capture the spiked shape of the fridge signal. The deficiency in the MCC score is a result of the reconstructed signal not being close enough to zero at times where the ground truth is zero. It should also be noted that the second unsupervised method, GSP, performs very well on average, with only the reconstruction of the passive signal failing severely in \Cref{fig:activity}, which is most likely caused by the method misinterpreting starting and ending points of events.

 \begin{table}[h]
\if\arxive1
\scriptsize
\fi
\centering
\begin{tabular}{|m{0.02\textwidth}m{0.05\textwidth}|m{0.03\textwidth}|m{0.03\textwidth}|m{0.05\textwidth}|m{0.03\textwidth}|m{0.03\textwidth}|m{0.03\textwidth}|}
\hline
 & & CSC	& GSP & FHMM	 & CO	& S2S	 & DSC 	\\
 \hline
synd1& $\text{MCC}_7$ 	 & 0.78 & 0.84 & 0.79 & 0.91 & \textbf{0.97} & 0.44\\
 &  	 $\text{F1}_7$ 	 & 0.78 & 0.84 & 0.81 & 0.92 & \textbf{0.97} & 0.47\\
 & 	 NMSE 	 & \textbf{0.01} & 0.07 & 0.12 & 0.09 & 0.04 & 0.36\\
 & 	 NMAE 	 & \textbf{0.13} & 0.96 & 0.38 & 0.30 & 0.79 & 0.95\\
synd2& $\text{MCC}_7$ 	 & 0.95 & 0.88 & 0.87 & 0.90 & \textbf{0.98} & 0.73\\
 &  	 $\text{F1}_7$ 	 & 0.95 & 0.88 & 0.86 & 0.91 & \textbf{0.98} & 0.72\\
 & 	 NMSE 	 & \textbf{0.002} & 0.11 & 0.14 & 0.11 & 0.05 & 0.21\\
 & 	 NMAE 	 & \textbf{0.17} & 0.63 & 0.66 & 0.57 & 0.53 & 0.92\\

\hline
\end{tabular}
\caption{Testing scores of the active energy usage for different methods applied to synd1 and synd2. Training and testing on different time frames of the same household. Bold indicates best result.}
\label{table:supervised_activity}
\end{table}

\begin{table}[h]
\if\arxive1
\scriptsize
\fi
\centering
\begin{tabular}{|m{0.02\textwidth}m{0.05\textwidth}|m{0.03\textwidth}|m{0.03\textwidth}|m{0.05\textwidth}|m{0.03\textwidth}|m{0.03\textwidth}|m{0.03\textwidth}|}
\hline
 & & CSC	& GSP & FHMM	 & CO	& S2S	 & DSC	\\
 \hline

synd1& $\text{MCC}_7$ 	 & 0.63 & 0.84 & 0.82 & \textbf{0.85} & 0.52 & 0.38\\
 &  	 $\text{F1}_7$ 	 & 0.63 & 0.84 & 0.82 & \textbf{0.87} & 0.54 & 0.40\\
 & 	 NMSE 	 & \textbf{0.001} & 0.07 & 0.04 & 0.16 & 0.06 & 0.18\\
 & 	 NMAE 	 & \textbf{0.30} & 0.96 & 0.32 & 0.44 & 0.75 & 0.97\\
synd2& $\text{MCC}_7$ 	 & \textbf{0.94} & 0.88 & 0.92 & 0.92 & \textbf{0.94} & 0.48\\
 &  	 $\text{F1}_7$ 	 & 0.94 & 0.88 & 0.92 & 0.92 & \textbf{0.95} & 0.47\\
 & 	 NMSE 	 & \textbf{0.003} & 0.11 & 0.11 & 0.11 & 0.08 & 0.52\\
 & 	 NMAE 	 & \textbf{0.06} & 0.63 & 0.57 & 0.61 & 0.56 & 0.99\\

\hline
\end{tabular}
\caption{Testing scores of the active energy usage for different methods applied to synd1 resp. synd2 with training on synd2 resp. synd1. Bold indicates best result.}
\label{table:unsupervised_activity}
\end{table}

\newcommand{\data}{synd2}
\newcommand{\w}{0.25}
\begin{figure}
\centering
\includegraphics[width=\w\textwidth]{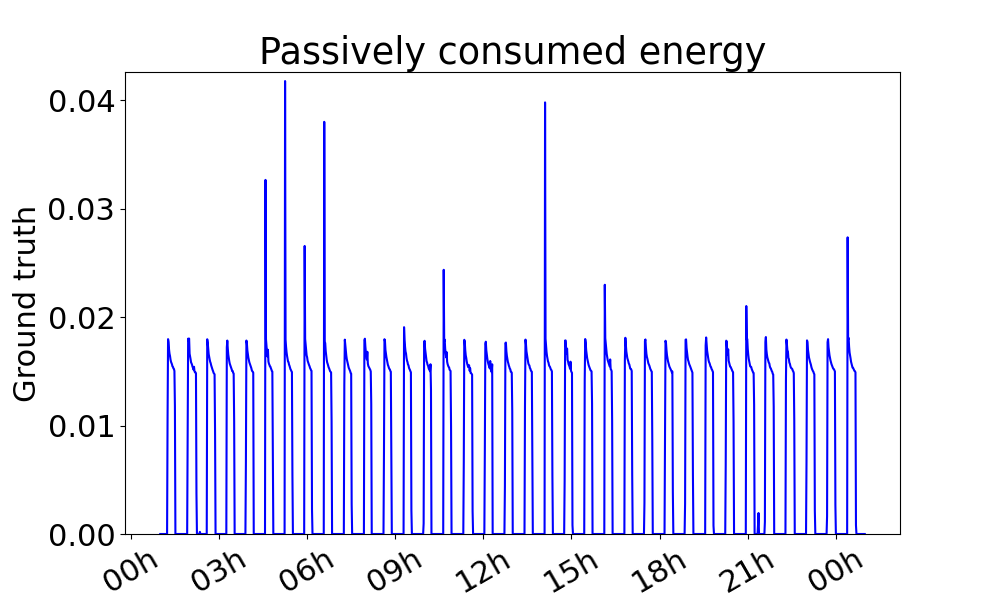}%
\includegraphics[width=\w\textwidth]{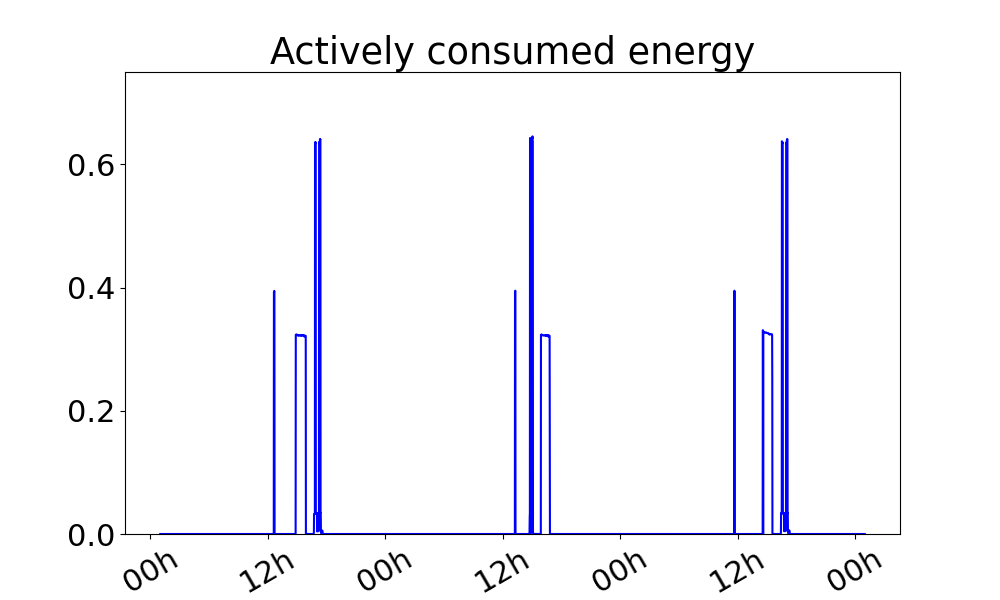}

\includegraphics[width=\w\textwidth]{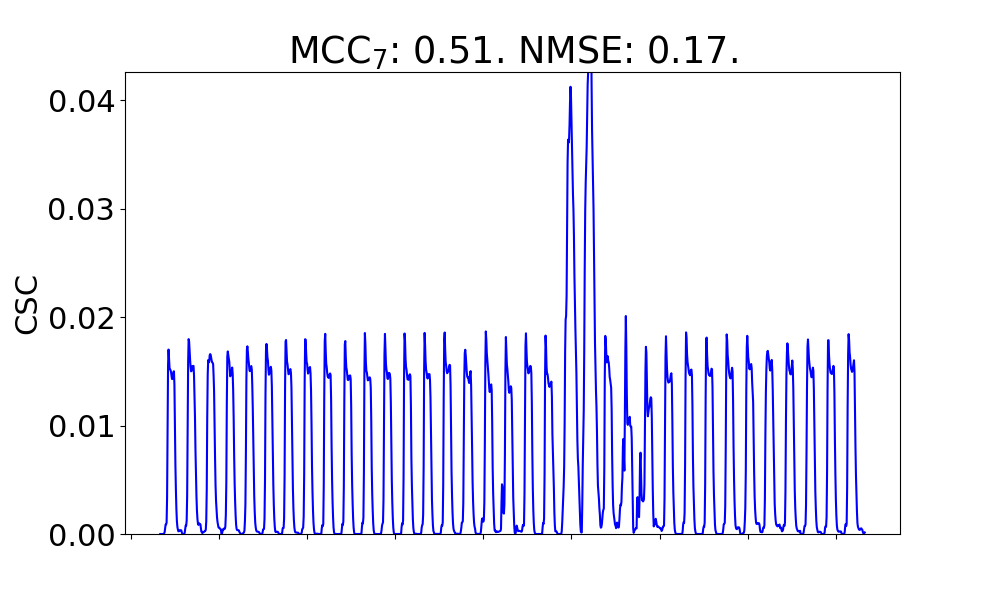}%
\includegraphics[width=\w\textwidth]{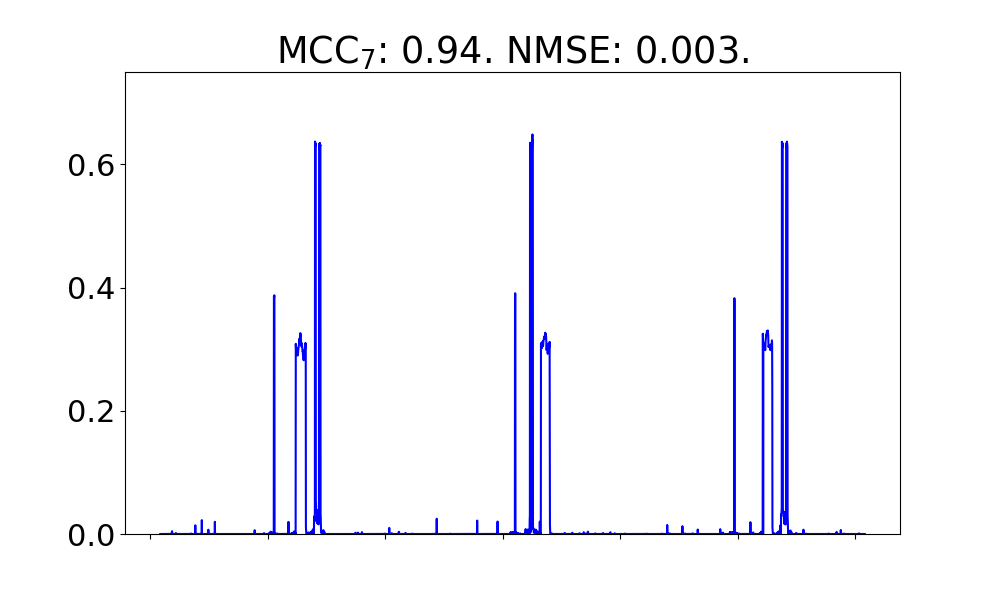}

\includegraphics[width=\w\textwidth]{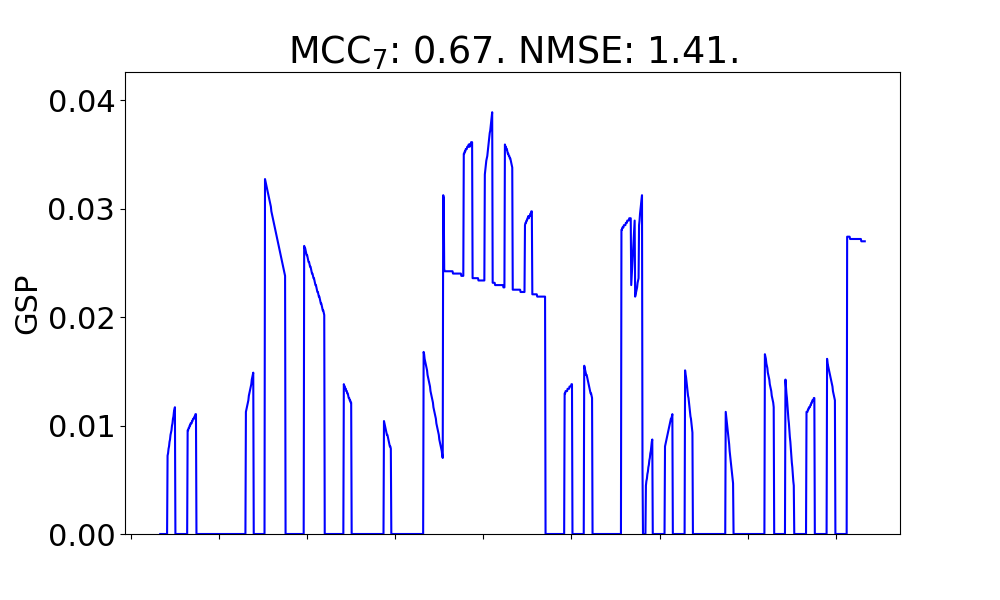}%
\includegraphics[width=\w\textwidth]{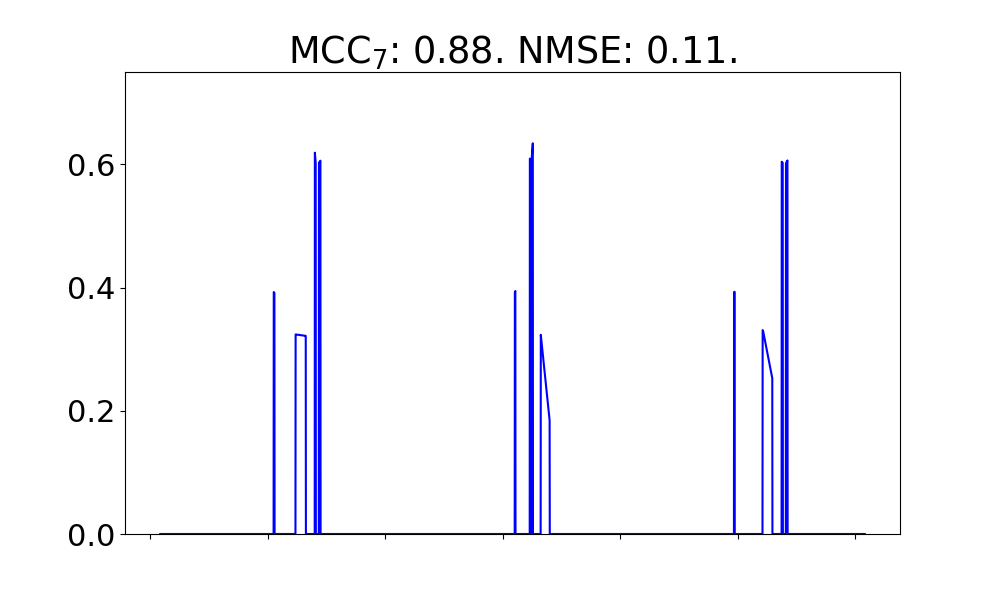}

\includegraphics[width=\w\textwidth]{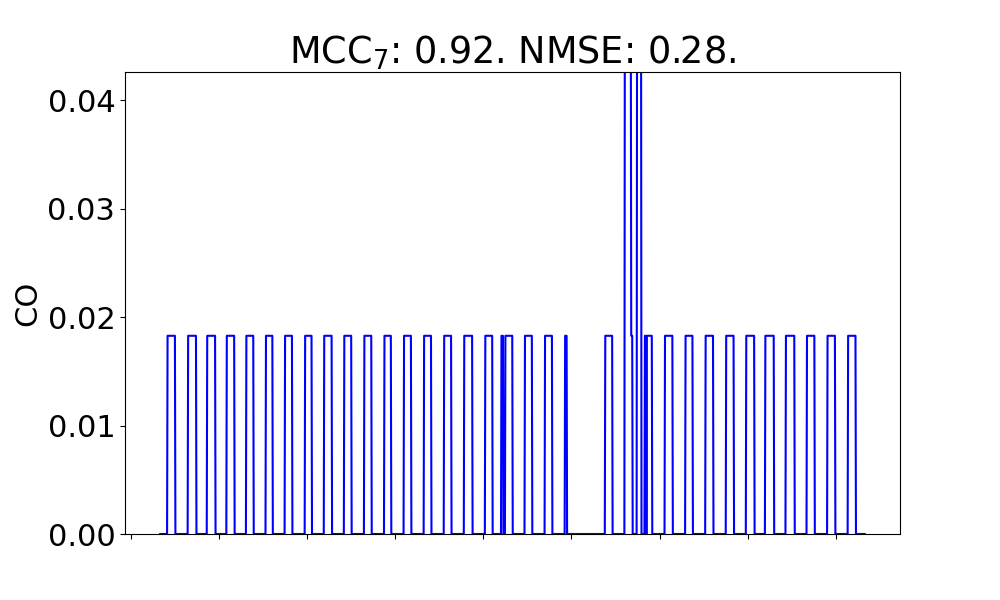}%
\includegraphics[width=\w\textwidth]{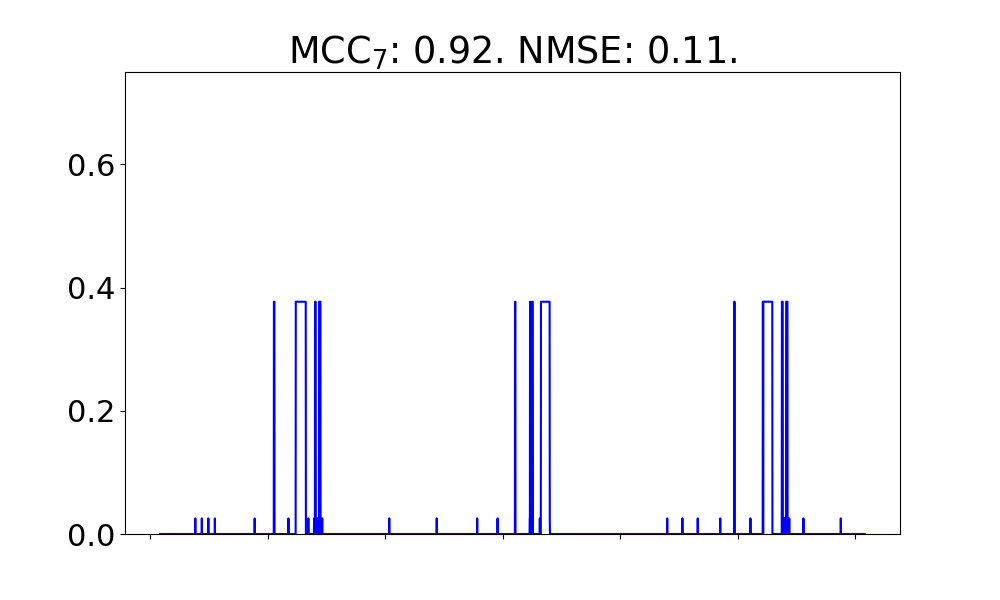}

\includegraphics[width=\w\textwidth]{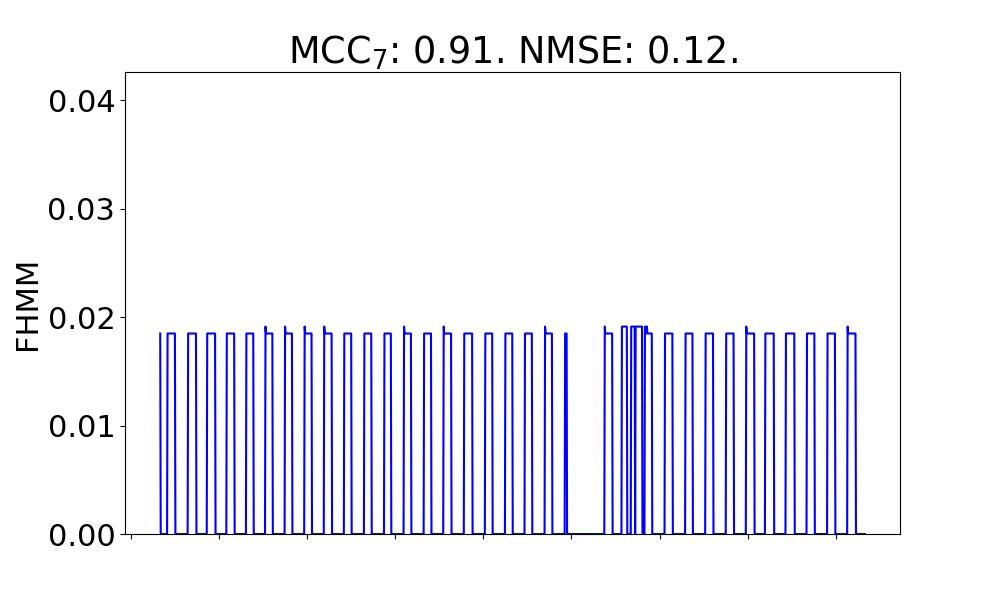}%
\includegraphics[width=\w\textwidth]{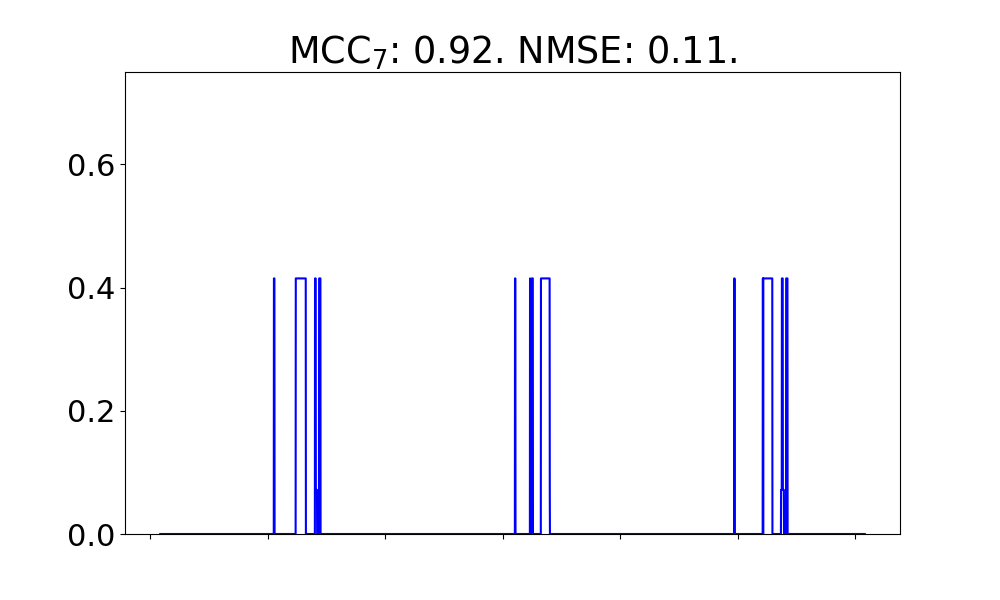}

\includegraphics[width=\w\textwidth]{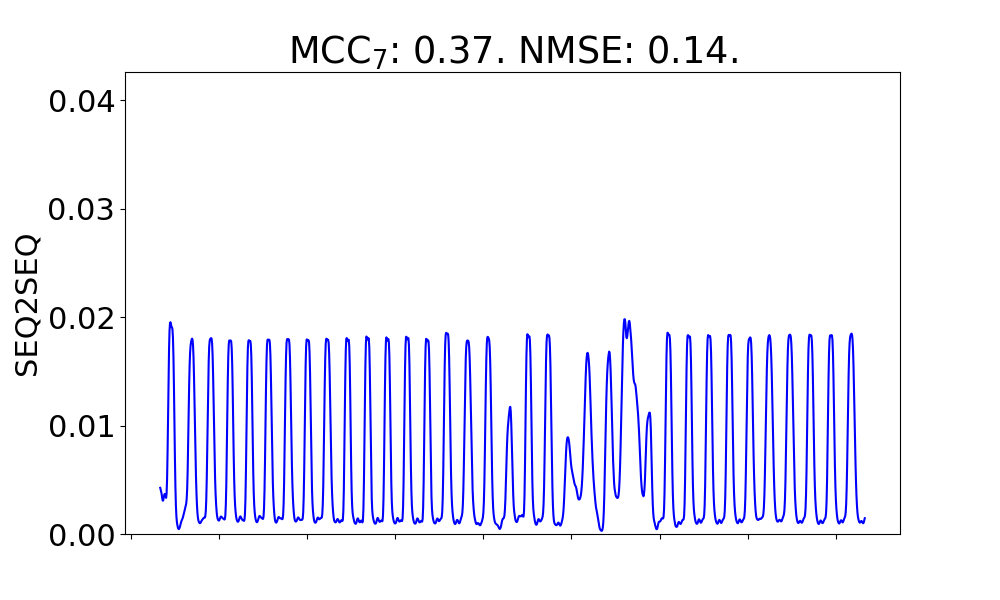}%
\includegraphics[width=\w\textwidth]{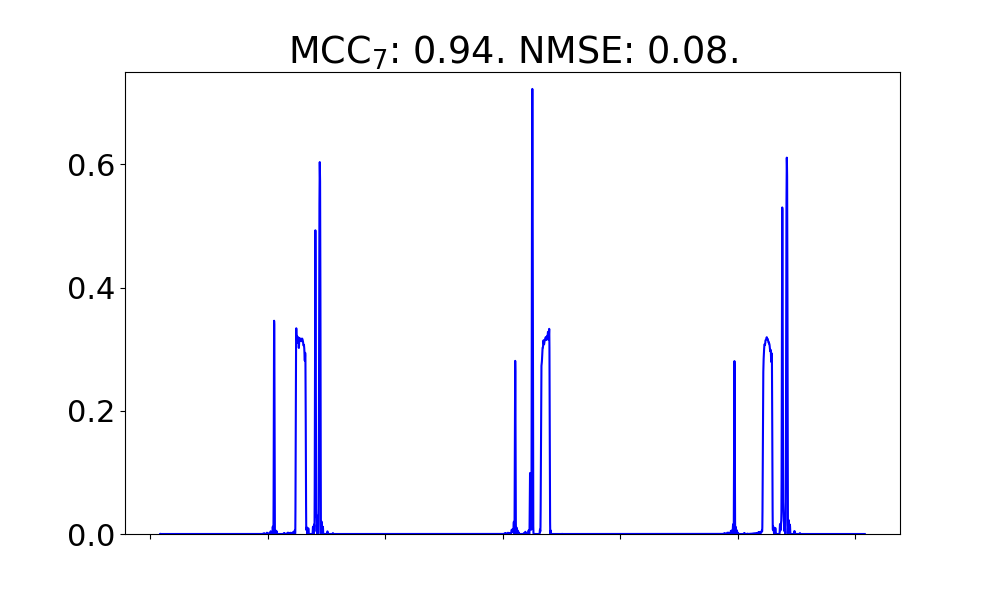}

\includegraphics[width=\w\textwidth]{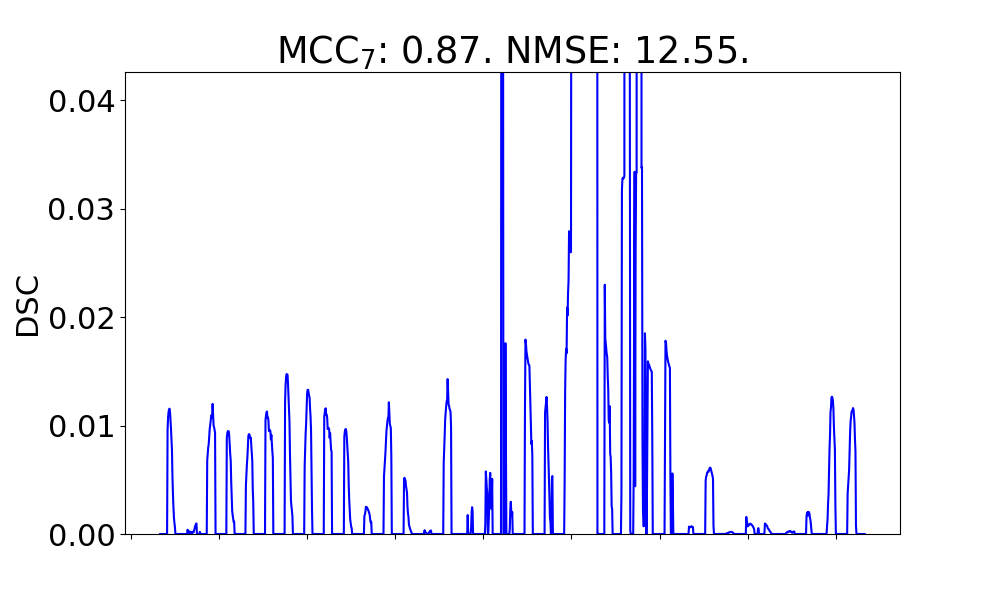}%
\includegraphics[width=\w\textwidth]{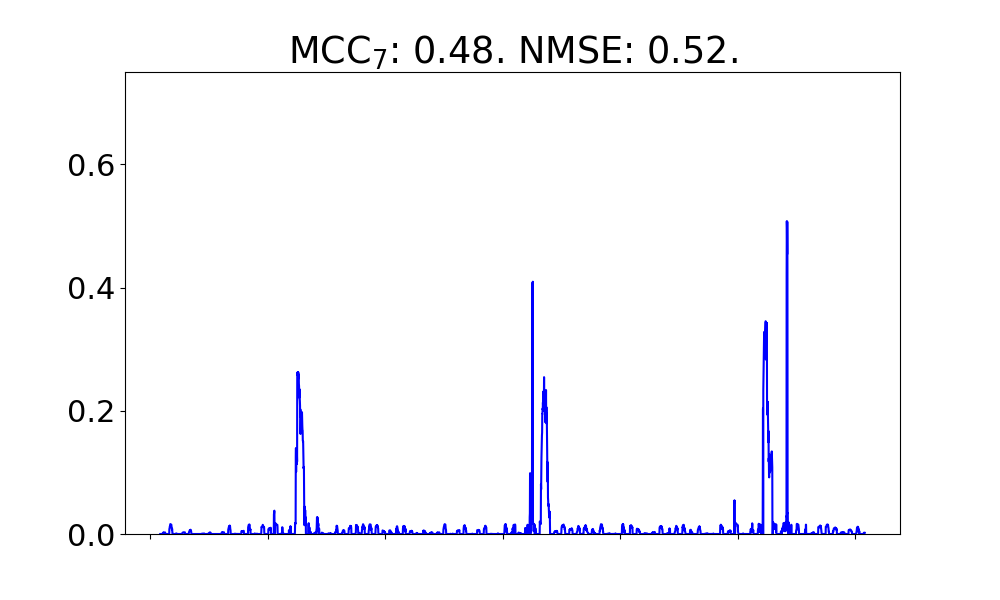}
\caption{Predictions for active and passive energy usage with different methods. Results for inference on synd2 with training on synd1. Left column: passive energy. Right column: active energy. Top row to bottom row: ground truth, predictions with CSC, GSP, CO, FHMM, S2S, DSC.}
\label{fig:activity}
\end{figure}

\subsubsection{Unsupervised multichannel disaggregation}\label{sec:exp_multi}
In this experiment, we investigate the performance of CSC and GSP applied to multichannel disaggregation, that is, inferring all individual devices from the given aggregate energy signal. We consider again the use case of no ground truth training data being available for the household of interest and estimate the CSC hyperparameters on a different dataset than the one used for testing. We cannot perform this experiment with the remaining methods, as those learn characteristics of specific devices, however, here the devices contained in the training and testing data differ. Therefore, this use case constitutes a setting, where only unsupervised methods are feasible. 
Quantitative results can be found in \Cref{table:unsupervised_multichannel}, and qualitative ones in \Cref{fig:multichannel}. Overall, CSC slightly outperforms GSP in terms of quantitative scores. In general, the reconstruction quality of active devices with CSC is dependent on the specific device signature. In \Cref{fig:multichannel}, we observe that devices with a more distinct and/or extended signature, such as the space heater or the iron, are reconstructed properly. The reconstruction is less precise for devices like the watercooker, with a signature composed of simple peaks of a very short duration.

\begin{table}[h]
\if\arxive1
\scriptsize
\fi
\centering
\begin{tabular}{|m{0.04\textwidth}|m{0.05\textwidth}|m{0.05\textwidth}|m{0.07\textwidth}|m{0.05\textwidth}|m{0.07\textwidth}|}
\hline
&	synd1	&Fridge	&Dishwasher	&Iron	&Watercooker\\
\hline
CSC &	$\text{MCC}_7$&\textbf{0.48}&	\textbf{0.74}&	0.88&	0.32\\
 &	$\text{F1}_7$&	\textbf{0.74}&	\textbf{0.73}&	0.88&	0.27\\
 &	NMSE&	\textbf{0.24}&	\textbf{0.13}&	\textbf{0.04}&	0.96\\
 &	NMAE&	\textbf{1.53}&	1.00&	\textbf{0.99}&	1.98\\
GSP&	$\text{MCC}_7$&	0.47&	0.55&	\textbf{0.96}&	\textbf{0.87}\\
 &	$\text{F1}_7$&	0.73&	0.47&	\textbf{0.96}&	\textbf{0.86}\\
 &	NMSE&	1.41&	0.14&	0.06&	\textbf{0.30}\\
 &	NMAE&	1.62&	\textbf{0.94}&	1.05&	\textbf{1.02}\\

\hline
\hline
&	synd2	&Fridge	&Hair dryer	&Space heater	&Washing machine\\
\hline
CSC&	$\text{MCC}_7$	&	0.48				&	\textbf{0.35}		&	0.74				&	\textbf{0.46}\\
 &	$\text{F1}_7$	&	0.73				&	\textbf{0.24}		&	0.72				&	\textbf{0.47}\\
 &	NMSE	&	\textbf{0.25}		&	\textbf{0.55}		&	\textbf{0.04}		&	\textbf{0.21}\\
 &	NMAE	&	1.12				&	\textbf{1.00}		&	\textbf{0.94}		&	\textbf{0.92}\\

GSP&	$\text{MCC}_7$	&	\textbf{0.66}		&	0.30				&	\textbf{0.77}	&	0.11\\
 &	$\text{F1}_7$	&	\textbf{0.80}		&	0.17				&	\textbf{0.76}	&	0.17\\
 &	NMSE	&	0.48				&	11.57			&	0.67				&	1.20\\
 &	NMAE	&	\textbf{1.00}		&	1.16				&	1.00				&	1.00\\

\hline
\end{tabular}
\caption{Testing scores of individual device signatures. Training CSC was done on the respective other household than testing. Bold indicates the better result.}
\label{table:unsupervised_multichannel}
\end{table}

\newcommand{\datat}{synd2}
\newcommand{\datao}{synd1}
\renewcommand{\w}{0.25}

\begin{figure}[t]
\centering
\subfloat[Results on synd1]{
\centering
\begin{tabular}[t]{c}%
\includegraphics[width=\w\textwidth]{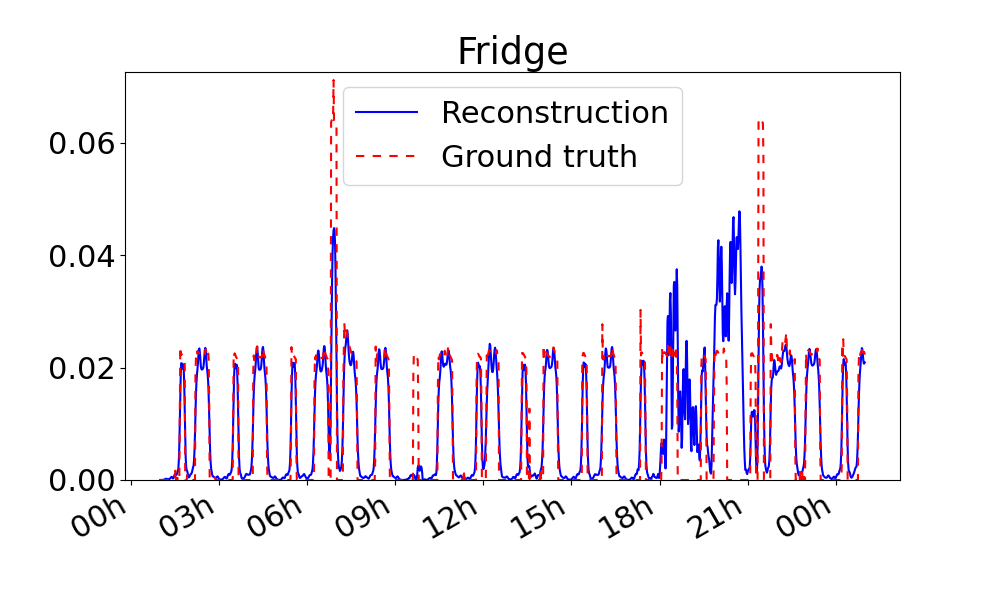}%
\includegraphics[width=\w\textwidth]{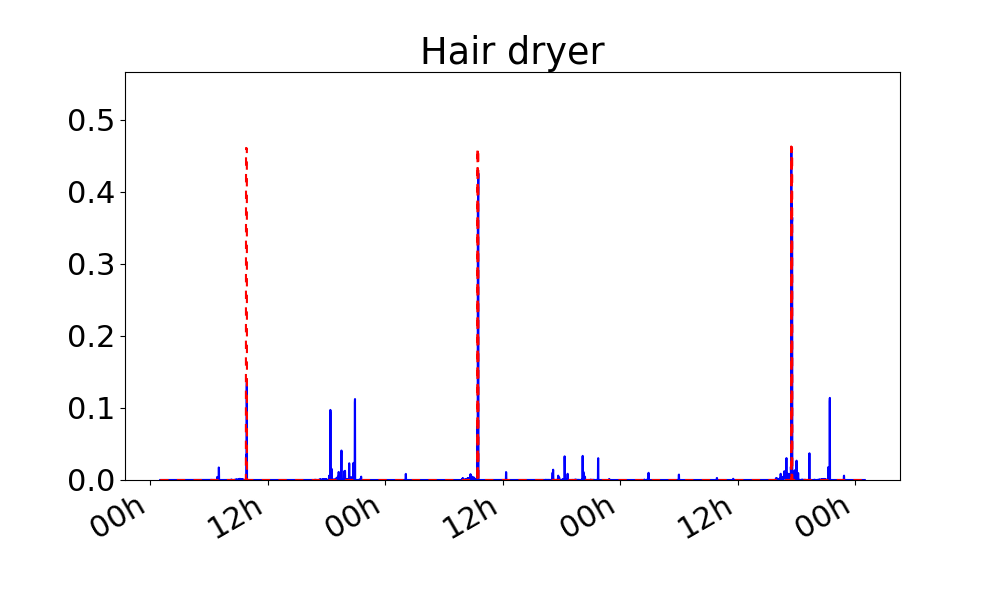}\\
\includegraphics[width=\w\textwidth]{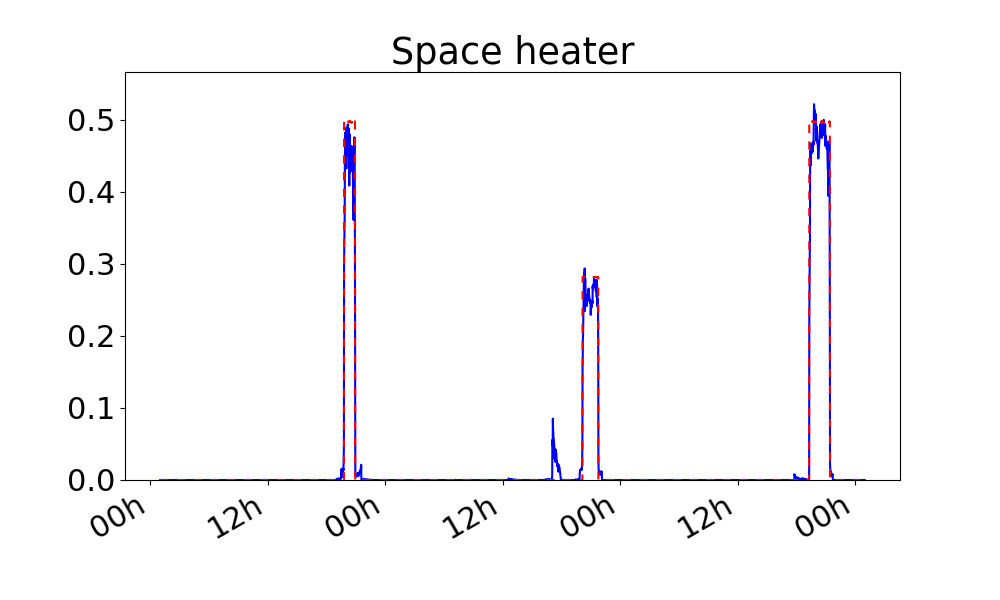}%
\includegraphics[width=\w\textwidth]{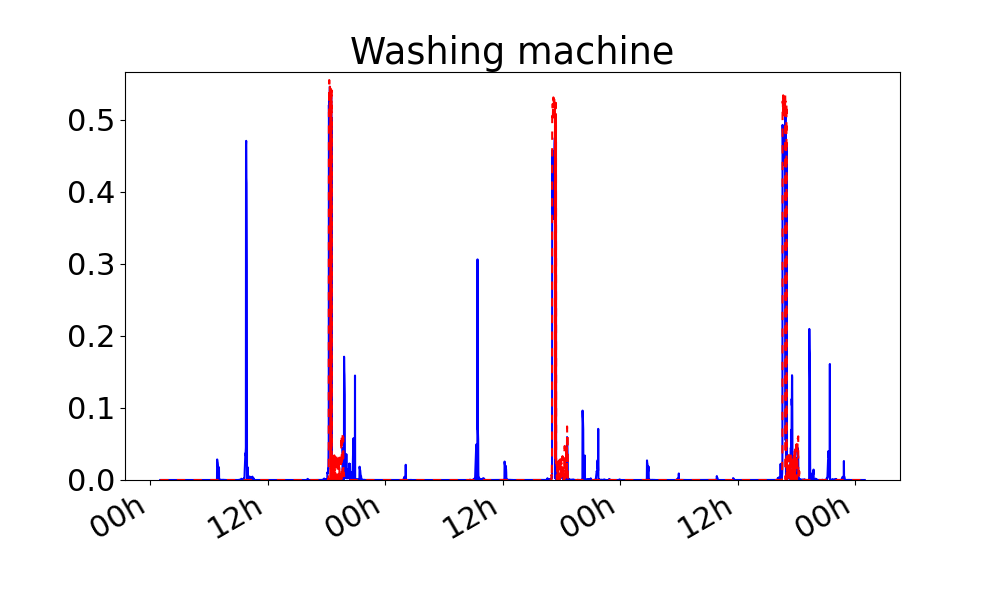}
 \end{tabular}}

\centering
\subfloat[Results on synd2]{
\centering
\begin{tabular}[t]{c}%
\includegraphics[width=\w\textwidth]{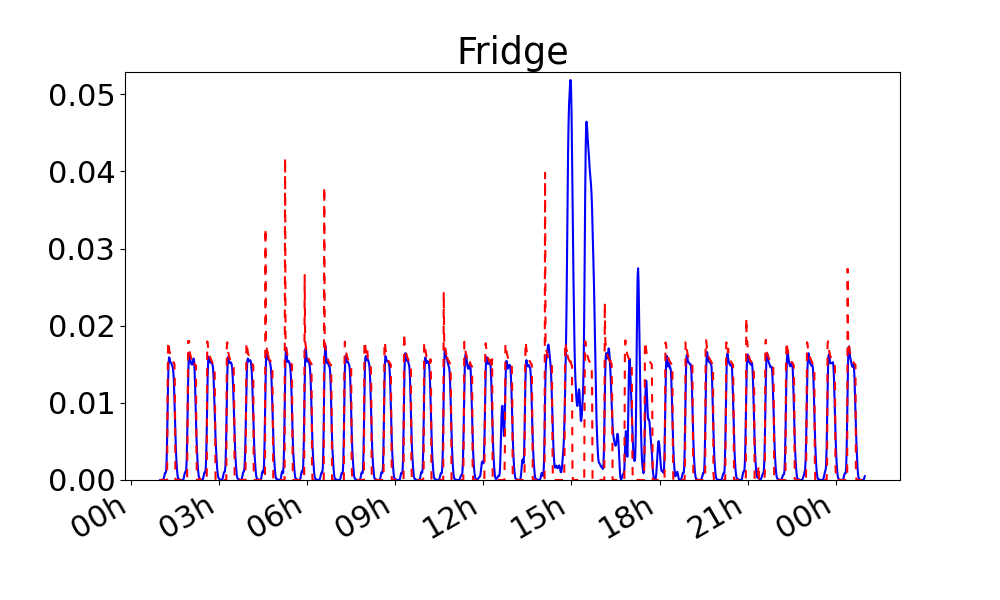}%
\includegraphics[width=\w\textwidth]{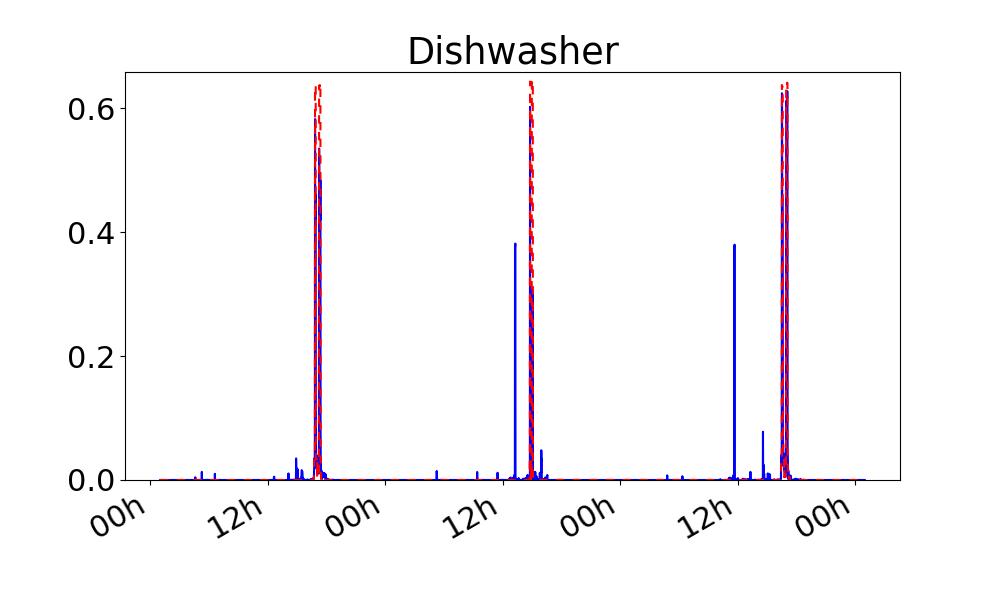}\\
\includegraphics[width=\w\textwidth]{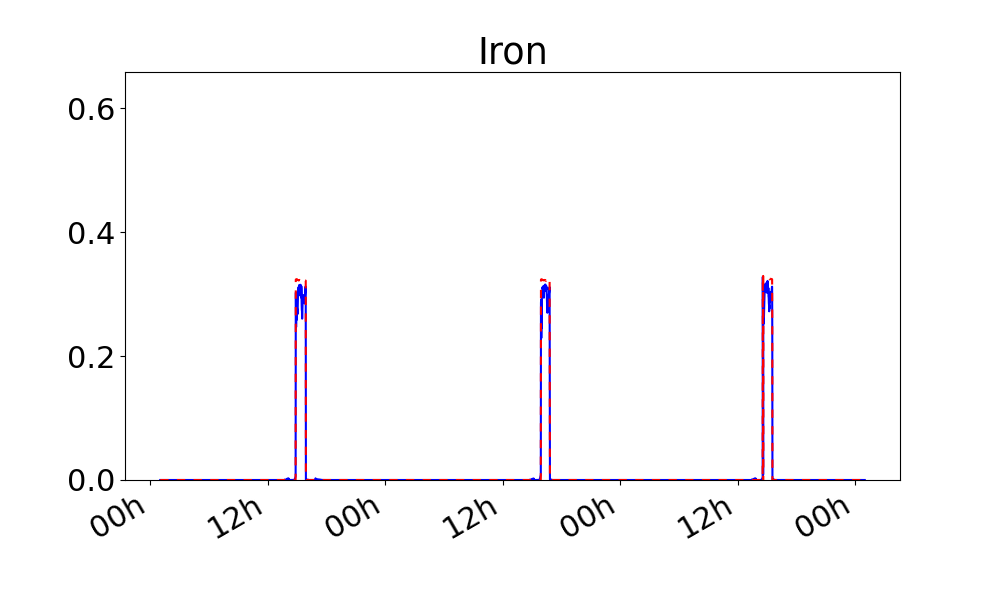}%
\includegraphics[width=\w\textwidth]{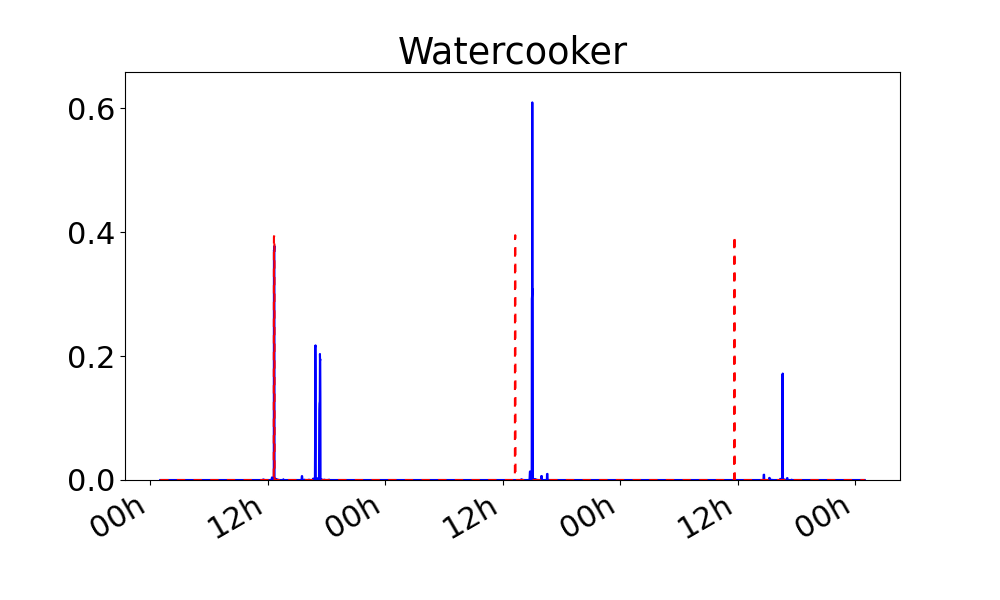}
 \end{tabular}}

\caption{Predictions of individual devices with CSC. Results for inference with training on the respective other data set.}
\label{fig:multichannel}
\end{figure}

\section{Conclusion}
In this work, we developed a new unsupervised method for energy disaggregation based on convolutional sparse coding with a focus on classifying energy consumption as active or passive. Prior training is not needed, as the method relies only on the choice of hyperparameters, which can be fixed for a broad class of input data. We employ the iPALM algorithm to minimize the used objective functional, ensuring strong convergence guarantees compared to many competing methods. As the results of our numerical experiments show, our unsupervised approach performs comparable to the state of the art in activity prediction in terms of classification scores, and yields superior performance in terms of NMSE and NMAE scores (see \Cref{table:supervised_activity,table:unsupervised_activity}). Further, it recognizes specific energy consumption signatures better (see \Cref{fig:activity}), and offers more flexibility due to the untrained approach. Applied to unsupervised multichannel disaggregation, our method performs slightly superior to the comparison method GSP \cite{zhao2016training}.

\newcommand{\vs}{-50pt}
\newcommand{\h}{2.5cm}
\bibliographystyle{abbrv}
\bibliography{references.bib}
\if\arxive0
\vspace{-40pt}
\begin{IEEEbiographynophoto}{Christian Aarset}
concluded his MSc in Mathematics at the University of Oslo, and his PhD in Mathematics at the University of Klagenfurt.%
\end{IEEEbiographynophoto}
\vspace{-40pt}
\begin{IEEEbiographynophoto}{Andreas Habring}
received his MSc in Mathematics from Graz University of Technology. He currently does his PhD in applied Mathematics at the University of Graz supervised by Martin Holler.%
\end{IEEEbiographynophoto}
\vspace{-40pt}
\begin{IEEEbiographynophoto}{Martin Holler} received his MSc and his PhD with a "promotio sub auspiciis" in Mathematics from the University of Graz. After research stays at the University of Cambridge, UK, and the Ecole Polytechnique, Paris, he currently holds a Assistant Professor position at the Institute of Mathematics and Scientific Computing of the University of Graz. %
\end{IEEEbiographynophoto}
\vspace{-40pt}
\begin{IEEEbiographynophoto}{Mario Mitter}
received his MSc and PhD in Physics and his MSc in Mathematics, all  "with honors", from the University of Graz. After research stays at the University of Heidelberg, DE, and the Brookhaven National Laboratory, Upton, NY, he switched to the field of Data Science in the private sector in 2019. %
\end{IEEEbiographynophoto}
\fi

\end{document}